\documentclass[a4paper,12pt]{amsart}
\usepackage{amsfonts,amsthm,amssymb,amsmath}
\usepackage{color}
\usepackage[colorlinks=true,linkcolor=blue,citecolor=blue]{hyperref}
\usepackage{fourier}

\newtheorem{theorem}{Theorem}[section]
\newtheorem{proposition}[theorem]{Proposition}

\newtheorem{lemma}[theorem]{Lemma}
\theoremstyle{definition}

\newtheorem{example}[theorem]{Example}

\DeclareMathOperator{\re}{Re}
\DeclareMathOperator{\spanned}{span}

\title{The Dirichlet-Bohr radius}

\author[D. Carando]{Daniel Carando}
\address{ Departamento de Matem\'atica, Fac. C. Exactas y Naturales, Universidad de Buenos Aires, Pab I,
Ciudad Universitaria, 1428, Buenos Aires, Argentina and IMAS - CONICET} \email{dcarando@dm.uba.ar}
\author[A. Defant]{ Andreas Defant}
\address{Institut f\"ur Mathematik, Universit\"at Oldenburg, D--$26111$, Oldenburg,
Germany} \email{defant@mathematik.uni-oldenburg.de}
\author[D. Garc\'{\i}a]{Domingo Garc\'{\i}a}
\address{Departamento de An\'{a}lisis Matem\'{a}tico,
Universidad de Valencia, Doctor Moliner $50$, $46100$ Burjasot (Valencia), Spain}
\email{domingo.garcia@uv.es}
\author[M. Maestre]{Manuel Maestre}
\address{Departamento de An\'{a}lisis Matem\'{a}tico,
Universidad de Valencia, Doctor Moliner $50$, $46100$ Burjasot (Valencia), Spain}
\email{manuel.maestre@uv.es}
\author[P. Sevilla-Peris]{Pablo Sevilla-Peris}
\address{Instituto Universitario de Matem\'atica Pura y Aplicada. Universitat Polit\`ecnica de Val\`encia, Valencia, Spain} \email{psevilla@mat.upv.es}
\thanks{The first author was partially supported by CONICET-PIP 0624,
UBACyT Grant 1-746 and  ANPCyT PICT 2011-1456. The  third, fourth and fifth
authors were supported by MICINN  MTM2011-22417. The third and fourth also by Prometeo II/2013/013. The fifth author was also partially supported by UPV-SP20120700}

\keywords{Dirichlet series, Bohr radius, Holomorphic functions} \subjclass[2010]{
  11M41, 30B50, 11M36}

\date{\today}

\begin{document}

\begin{abstract}
\noindent   Denote by $\Omega(n)$  the number of prime divisors of $n \in \mathbb{N}$
(counted with multiplicities).
 For $x\in \mathbb{N}$ define the Dirichlet-Bohr radius $L(x)$ to be  the best  $r>0$ such that for every finite Dirichlet polynomial $\sum_{n \leq x} a_n n^{-s}$ we have
$$
\sum_{n \leq x}  |a_n| r^{\Omega(n)} \leq  \sup_{t\in \mathbb{R}} \big|\sum_{n \leq x} a_n n^{-it}\big|\,.
$$
We prove that  the asymptotically correct order of $L(x)$ is $ (\log x)^{1/4}x^{-1/8} $.
Following  Bohr's vision our proof links the estimation of $L(x)$  with classical Bohr radii for holomorphic functions in several variables. Moreover, we suggest a general setting which   allows to translate various results on Bohr radii in a systematic way into results on Dirichlet-Bohr radii, and vice versa.
\end{abstract}

\maketitle

\section{Introduction}
\noindent The study of problems on absolute convergence of Dirichlet series (of the form $\sum_{n} a_{n} n^{-s}$, where $s$ is a complex variable) led H.~Bohr to relate properties on absolute convergence with properties of boundedness (on the right half plain)
of the holomorphic function defined by the Dirichlet series. One of his first results in this direction is the following inequality \cite[Satz~XIII]{Bo13_Goett}: for every Dirichlet series of the form $\sum_{p \text{ prime}} a_p p^{-s}$ we have
\begin{equation} \label{Bohr's inequality}
\sum_{p \text{ prime}} |a_p| \leq \sup_{\re s >0}\Big| \sum_{p \text{ prime}} a_p p^{-s} \Big|\,.
\end{equation}
In his research \cite{Bo13_Goett,Bo13} he then established  a close relationship between Dirichlet series and power series in infinitely many variables (this relationship was presented in a modern, systematic way much later by Hedenmalm, Lindqvist and Seip
\cite{HeLiSe97}). Bohr then looked at holomorphic functions and proved his well known power series theorem \cite{Bo14}: for every holomorphic function $f$ on the open unit disc $\mathbb{D}$ we have
\begin{equation} \label{Bohr's power series theorem}
\sum_n  \big|\frac{f^{(n)}(0)}{n!}\big|  \frac{1}{3^n}  \leq \|f\|_\infty\,,
\end{equation}
and that here moreover the number $1/3$ is optimal. As a simple consequence of the maximum modulus principle, it can be seen that for each Dirichlet series $\sum_{n} a_{2^n} 2^{-ns}$ we have
\[
\sup_{z \in \mathbb{D} }
\Big|
\sum_{n} a_{2^n} z^n
\Big|
=
\sup_{\re s >0}
\Big|
\sum_{n} a_{2^n} 2^{-ns}\Big| \,.
\]
Hence \eqref{Bohr's power series theorem} can be reformulated as follows: for each Dirichlet series  $ \sum_{n} a_{2^n} 2^{-ns}$
\begin{equation} \label{motherAAA}
\sum_{n} \Big|a_{2^n} \frac{1}{3^n}\Big|  \leq   \sup_{\re s >0}\Big| \sum_{n} a_{2^n} 2^{-ns} \Big|\,.
\end{equation}
The work of Dineen and Timoney \cite{DiTi89} renewed the interest on Bohr's theorem and Boas and Khavinson \cite{BoKh97} defined the $n$-dimensional Bohr radius $K_{n}$ to be the best $0<r<1$ such that
\[
\sum_{\alpha \in \mathbb{N}_{0}^{n}}  \Big\vert  \frac{\partial^{\alpha} f(0)}{\alpha !} \Big\vert r^{\vert \alpha \vert}
\leq \sup_{z \in \mathbb{D}^{n}} \Big\vert \sum_{\alpha \in \mathbb{N}_{0}^{n}}   \frac{\partial^{\alpha} f(0)}{\alpha !} z^{\alpha} \Big\vert \,,
\]
for every bounded, holomorphic function $f$ on $\mathbb{D}^{n}$. That was the starting point of a long search on the optimal asymptotic behaviour of $K_{n}$ as $n$ grows that was finally closed in \cite{DeFrOrOuSe11} and \cite{BaPeSe13} (see Section~\ref{generales} for more details).\\

Because of the link between Dirichlet series and power series,  each result in either framework has an immediate translation into the other. This is of course the case with the behaviour of $K_{n}$ (a fact which is stated in more detail in Example~\ref{motherAA}). But,
as it happens, what is natural in one side may not be as natural in the other; and while taking $n$ variables (or, equivalently, $n$-dimensional spaces) is natural in the side of holomorphic functions, in the side of Dirichlet series we would rather take finite sums of (the first) $n$~terms.\\
So, inspired by the Bohr radius for holomorphic functions, our main aim in this note is to determine, for each $x \geq 2$, the best  $r=r(x) \ge 0$ such that for every finite Dirichlet polynomial $ \sum_{n \leq x} a_{n} n^{-s} $ of length $x$
\[
\sum_{n \leq x} \vert a_{n} \vert r^{\Omega(n)} \leq \sup_{\re s >0}\Big| \sum_{n \leq x} a_{n} n^{-s} \Big|\,,
\]
where $\Omega(n)$ denotes  the number of prime divisors of $n \in \mathbb{N}$ (counted with multiplicities).
We do this in our main result Theorem~\ref{main}, that gives  the asymptotically correct order of this best radius.

We then take a general point of view and, for a given subset $J$ of  $\mathbb{N}$, we define the {\it Dirichlet-Bohr radius} $L(J)$ of $J$ to be the  best $ r=r(J)\ge 0 $ such that for every Dirichlet series $ \sum_{n \in J} a_{n} n^{-s} $ convergent on
the open half-plane $[\re  s>0]$, we have
\begin{equation}\label{ladefinition}
\sum_{n \in J} \vert a_{n} \vert r^{\Omega(n)} \leq \sup_{\re s >0}\Big| \sum_{n \in J} a_{n} n^{-s} \Big| \,.
\end{equation}
\noindent With this, denoting by $P$ the set of prime numbers, \eqref{Bohr's inequality} and \eqref{motherAAA} can be rephrased as
\begin{align} \label{o'flynn}
L(P)=1  && \text{and} && L\big(\big\{ 2^k \,|\, k \in \mathbb{N}  \big\}\big) = \frac{1}{3}\,.
\end{align}
Then, Theorem~\ref{main} gives the correct asymptotic order of $L(\{ n \in \mathbb{N}\,| 1 \leq n \leq x \})$. We will see that, following an idea of H. Bohr based on Diophantine approximation, this study can be extended to other sets $J$ of indices.\\

Finally, we mention another estimate which seems of relevance when motivating our results: For every
$\varepsilon >0$ there is $C=C(\varepsilon) \ge 1$ such that for every $x$ and finite Dirichlet polynomial
$ \sum_{n \leq x} a_{n} n^{-s} $
\begin{equation} \label{monster}
\sum_{n \leq x} \vert a_n \vert  \frac{  e^{ \big(\frac{1}{\sqrt{2}} - \varepsilon \big)\sqrt{\log n\log \log n}}}{n^{1/2}}\,
\,
\leq\,  C\,\sup_{\re s >0}\Big| \sum_{n \leq x} a_{n} n^{-s} \Big|\,.
\end{equation}
This result is under several different aspects optimal, and  it is the final outcome of a long series of results due to\cite{BaCaQu06,Br08,DeFrOrOuSe11,KoQu01,Qu95,QuQu13}. Our main result, Theorem~\ref{main}, can be considered to be a relative  of \eqref{monster}.

\subsection{Notations} As we have already mentioned, $\Omega (n)$ denotes, for $n \in \mathbb{N}$, the number of prime divisors of $n$, counted with their multiplicity. We denote by $(p_{n})_{n}$ the sequence of prime numbers. The set
of multiindices $\alpha$ that eventually become $0$ is denoted by $\mathbb{N}_{0}^{(\mathbb{N})}$. For $\alpha= (\alpha_{1}, \ldots , \alpha_{k}, 0, \ldots)$ we write $p^{\alpha} = p_{1}^{\alpha_{1}} \cdots p_{k}^{\alpha_{k}}$ and
$\vert \alpha \vert = \alpha_{1} + \cdots + \alpha_{k}$.\\
Along this note $\pi$ denotes the prime counting function, i.e.,  $\pi(x)$ is the number of prime numbers less than or equal to $x$.\\
Given two real functions $f$ and $g$ we write $f(x) \ll g(x)$ if there exist a constant $C>0$ such that $f(x) \leq C g(x)$ for every $x$. If $f(x) \ll g(x)$ and $g(x) \ll f(x)$  we write $f(x) = O(g(x))$.\\
For each $N$ we denote by $H_{\infty}(\mathbb{D}^{N})$ the space of bounded, holomorphic functions on $\mathbb{D}^{N}$. If $f \in H_{\infty}(\mathbb{D}^{N})$ and $\alpha \in \mathbb{N}_{0}^{N}$ we write $c_{\alpha}(f)
= \frac{\partial^{\alpha} f(0)}{\alpha !}$, the $\alpha$-th coefficient of the monomial expansion.

\section{Main result}
\noindent For any $x \geq 2$, we write
$$
L(x) = L\big(\big\{ n \in \mathbb{N}\,\big| \, 1 \leq n \leq  x \big\}  \big)\,,
$$
where $L$ is defined in \eqref{ladefinition}, and call this number the $x$-th Dirichlet-Bohr radius. The main result of this note then reads as follows.
\begin{theorem} \label{main} We have
\[
L(x) = O \left( \frac{\sqrt[4]{\log x}}{x^{1/8}}   \right)\,.
\]
In particular, there is a universal constant $C>0$  such that
 \[
\sum_{n \leq x} \vert a_{n}\vert \left(\frac{C\sqrt[4]{\log n}}{n^{1/8}}\right)^{\Omega(n)} \leq
\sup_{\re s >0}\Big| \sum_{n \leq x} a_{n} n^{-s} \Big|\,
\]for every $x \geq 2$ and every
 finite Dirichlet polynomial $ \sum_{n \leq x} a_{n} n^{-s} $.
\end{theorem}

\vspace{2mm}

\noindent The rest of this section is devoted to the proof of this result.

\subsection{Reduction I}
We start with a device which reduces the estimation of  Dirichlet-Bohr radii $L(x)$ to the estimation of their {\it homogeneous
parts $L_{m}(x)$} which we are going to define now. For $x \geq 2$ define the finite dimensional Banach space
\begin{gather*}
\mathcal{H}^{(x)}_\infty: = \Big\{ D=\sum_{n=1}^\infty a_n n^{-s}\,\, \Big| \,\, a_n \neq 0  \,\,\, \text{ only if  }
\,\,\, n \leq x  \Big\}
\\
\|D\|_\infty = \sup_{t \in \mathbb{R}} \Big|\sum_{n \leq x} a_n \frac{1}{n^{it}}  \Big|=
\sup_{\re s >0} \Big|\sum_{n \leq x} a_n \frac{1}{n^{s}}  \Big|
\end{gather*}
together with its closed subspace
\[
\mathcal{H}^{(x,m)}_\infty: =
\Big\{
\sum_{n=1}^{\infty} a_n n^{-s}\,\, \Big| \,\, a_n \neq 0  \,\,\, \text{ only if  }
\,\,\, n \leq x \text{ and } \Omega(n)=m
\Big\}\,.
\]
Then
\[
L(x) = \sup \Big\{ 0 \leq r\leq 1 \, \big| \, \forall  D \in \mathcal{H}^{(x)}_\infty:
\sum_{n \leq x} \vert a_{n} \vert r^{\Omega(n)} \leq\|D\|_\infty\,
 \Big\}\,,
\]
and therefore  for $m \in \mathbb{N}$ we define the $m$-homogeneous $x$-th Dirichlet-Bohr radius by
\begin{equation} \label{pentagono}
L_{m}(x)  : =  \sup \Big\{0 \leq r \leq 1 \,\, \Big| \,\,\forall  D \in \mathcal{H}^{(x,m)}_\infty
:\, \sum_{n \leq x}  |a_n|  \leq r^{-m}  \|D\|_\infty \Big\}\,.
\end{equation}
The following result is the announced \textit{reduction theorem}.

\begin{proposition} \label{reduction} With the previous notation, we have
\[
\frac{1}{3}  \inf_m  L_{m}(x) \,\,\leq\,\, L(x) \leq \,\, \inf_m  L_{m}(x) \,\,\, \text{ for all }\,\, x \geq 2 \,.
\]
\end{proposition}
We start with a reformulation  in terms of holomorphic functions. Note that if $n=p^{\alpha}$ and $1 \leq n \leq x$ then clearly $\alpha$ has at most the first $\pi(x)$ coordinates different from zero; in other words $\alpha \in \mathbb{N}_{0}^{\pi(x)}$.
Then, by Bohr's fundamental lemma (see \cite{QuQu13}) we know that for every  finite Dirichlet polynomial $\sum_{n \leq x}  a_n n^{-s}$ we have
\begin{align} \label{bohr}
 \sup_{t\in \mathbb{R}}
  \Big|\sum_{n \leq x} a_n n^{-it} \Big|= \sup_{z \in \mathbb{D}^{\pi(x)}} \Big|
 \sum_{\substack{\alpha \in \mathbb{N}_0^{\pi(x)}\\ 1 \leq p^\alpha \leq x}} a_{p^\alpha} z^\alpha \Big|\,.
\end{align}
With this identity in mind we define the Banach space
\begin{align*}
H^{(x)}_\infty :=  \Big\{ f \in H_\infty(\mathbb{D}^{\pi(x)}) \,\,\Big| \,\, c_\alpha(f)\neq 0 \,\,\, \text{ only if }
\,\,\,   p^\alpha \le x
    \Big\}\,,
\end{align*}
(the norm clearly given by the right  side of  \eqref{bohr}) and its closed subspace
\begin{align*}
H^{(x,m)}_\infty :=  \Big\{ f \in H_\infty(\mathbb{D}^{\pi(x)}) \,\,\Big| \,\, c_\alpha(f)\neq 0 \,\,\, \text{ only if }
\,\,\,   p^\alpha \le x \text{ and } |\alpha|=m\,
    \Big\}\,.
\end{align*}
 Identifying  Dirichlet series  $\sum_{n \leq x} a_n n^{-s} $
 with functions $\sum_{\substack{\alpha \in \mathbb{N}_0^{\pi(x)}\\ 1 \leq p^\alpha \leq x}} a_{p^\alpha} z^\alpha$ we then obtain  the following isometric equalities
\[
\mathcal{H}^{(x)}_\infty = H^{(x)}_\infty
\,\,\, \,\,
\text{ and }
\,\,\, \,\,
\mathcal{H}^{(x,m)}_\infty = H^{(x,m)}_\infty\,,
\]
and this in turn shows  that
\begin{equation} \label{count1}
L(x) =
\sup \Big\{0 \le r \leq 1 \, \,\Big|\,\,
\, \forall f \in H^{(x)}_\infty :\,
\sum_{\substack{\alpha \in \mathbb{N}_0^{\pi(x)}\\ 1 \leq p^\alpha \leq x}}
\big| c_\alpha(f) \big| r^{|\alpha|} \leq  \|f\|_\infty \Big\}
\,,
\end{equation}
and
\begin{equation} \label{count2}
L_{m}(x)
 =  \sup
 \Big\{
 0 \leq r \leq 1 \,\, \Big| \,\,\forall  f \in H^{(x,m)}_\infty
:\,
\sum_{\substack{1 \leq p^\alpha \le x\\|\alpha|=m}}
\big| c_\alpha(f)
\big| \leq  \,r^{-m}\,\|f\|_\infty \Big\}\,.
\end{equation}

\begin{proof}[Proof of Proposition \ref{reduction}]
The proof of the upper estimate is obvious, and for the proof of the lower estimate we follow  \cite[Section~2]{DeGaMa03}.
Fix $f \in H^{(x)}_\infty$ with $\|f\|_\infty \leq 1$, and write for its $m$-homo\-geneous part
\[
f_m(\omega) = \sum_{\substack{1 \leq p^\alpha \le x\\|\alpha|=m}} c_\alpha(f) \omega^\alpha \,,\,\,\omega \in  \mathbb{D}^{\pi(x)}\,;
\]
obviously, $f_m \in H^{(x,m)}_\infty$ and using  Cauchy inequalities we see that $\|f_m\|_\infty \leq 1$ for all $m$. We fix now some $z_0 \in \mathbb{D}^{\pi(x)}$ and $\theta \in \mathbb{T}$ such that $|c_0(f)|= \theta c_0(f)$, and define
\begin{align*}
&
g: \mathbb{D} \rightarrow \mathbb{C}\,,\,\, g(\omega) := f(\omega z_0)=
\sum_{m=1}^\infty
f_m(z_0) \omega^m\,,
\\&
h: \mathbb{D} \rightarrow \mathbb{C}\,, \,\, h := 1 - \theta g\,,
\end{align*}
Since $\|g\|_\infty \leq 1$, we have that $\re  h \ge 0$ on $\mathbb{D}$, and by Caratheodory's theorem (for an elementary proof, see \cite[Lemma~1.1]{Ai05})  we have for all $m$
\begin{equation} \label{enzo}
\big|f_m(z_0)\big|
=
\frac{h^{(m)}(0)}{m!} \leq 2 \re h(0) = 2 (1 - |c_0(f)|)\,.
\end{equation}
We take now some $r < \inf_m  L_{m}(x)$. Then for all $z \in \mathbb{D}^{\pi(x)} $ and all $m$ we have by \eqref{count2} and \eqref{enzo}
\[
 \sum_{\substack{1 \leq p^\alpha \le x\\|\alpha|=m}} \big|c_\alpha(f) (\frac{r}{3}z)^\alpha  \big|
 \leq \frac{1}{3^m} \| f_m\|_\infty
 \leq
 \frac{1}{3^m}2 (1 - |c_0(f)|) \,,
\]
and hence for all $z \in \frac{r}{3}\mathbb{D}^{\pi(x)}$
\[
 \sum_{\substack{1 \leq p^\alpha \le x}} \big|c_\alpha(f) z^\alpha  \big|
 \leq |c_0(f)| + \sum_{m=1}^\infty \frac{1}{3^m}2 (1 - |c_0(f)|) =1\,.
\]
The conclusion now follows from \eqref{count1}\,.
\end{proof}

\subsection{The tool}
The following proposition is our main tool --  a reelaboration of a result due  to Balasubramanian, Calado, and Queff\'{e}lec \cite[Theorem~1.4]{BaCaQu06} (see also \cite[Theorem~4.2]{DeScSe14}).

\begin{proposition} \label{balasubramanian}
Let $m\geq 2$ and $\kappa>1$. There exists  $C(\kappa)>0$  such that for every $m$-homoge\-neous Dirichlet polynomial $D=\sum_{n \leq x} a_{n} n^{-s}$ in $\mathcal{H}^{(x,m)}_\infty$ we have
\[
\sum_{n \leq x}|a_n|\frac{(\log n)^{\frac{m-1}{2}}}{n^{\frac{m-1}{2m}}}\leq C(\kappa) m^{m-1}(2\kappa)^{m}\|D\|_\infty\, .
\]
\end{proposition}
Our proof follows from a careful analysis of the original proof of \cite{BaCaQu06}, that allows us to obtain the constant $C(\kappa) m^{m-1}(2\kappa)^{m}$, smaller than the
original one. Since this fact is essential for our purpose, we for the sake of completeness prefer to add the proof. 
Every $m$-homogeneous polynomial in $n$ variables admits two possible representations:
\[
P(z) =  \sum_{\substack{\alpha \in \mathbb{N}^n \\ \vert \alpha \vert =m}} c_\alpha z^\alpha= \sum_{1 \leq j_{1} \leq \cdots \leq j_{m} \leq n} c_{j_{1} , \ldots , j_{m}} z_{j_1}\cdot \ldots \cdot z_{j_m}\,,\,\, \text{ for } z \in \mathbb{C}^n.
\]
We need  the  following lemma \cite[page 492]{DeFrOrOuSe11} (see also \cite[Lemma~4.3]{DeScSe14} or \cite[Lemma~2.6]{BaDeFrMaSe}).

\begin{lemma} \label{Fred2}
Let $n\geq 1$, $m\geq 1$ and  $\kappa>1$. Then there exists $C(\kappa)>0$ such that,
for every $m$-homogeneous polynomial on $\mathbb C^n$ we have
\[
\sum_{j_{m} =1}^{n} \Big(\sum_{1 \leq j_{1} \leq \cdots \leq j_{m}}\vert c_{j_{1} , \ldots , j_{m}} \vert^2\Big)^{\frac12}
\leq C(\kappa) \big( 2 \kappa \big)^m \sup\{|P(z)|: z\in \mathbb{D}^n\} \,.
\]
\end{lemma}

\begin{proof}[Proof of Proposition \ref{balasubramanian} ]
We begin by fixing some finite Dirichlet polynomial  $$D=\sum_{n \leq x} a_{n} n^{-s}\in  \mathcal{H}^{(x,m)}_\infty .$$ Now we define the following $m$-homogeneous polynomial  in $\pi(x)$ variables
\[
P(z) = \sum_{1 \leq j_{1} \leq \cdots \leq j_{m}\le\pi(x)} c_{\mathbf{i}} z_{j_1}\cdot \ldots \cdot z_{j_m}\,,\,\, z \in \mathbb{C}^{\pi(x)}\,,
\]
where   $c_{j_{1} \ldots j_{m}} = a_{n}$ for $1 \leq n=p_{j_{1}} \cdots p_{j_{m}} \leq x$ and $0$ otherwise.  Then
\begin{multline*}
 \sum_{n \leq x} \vert a_{n} \vert \frac{(\log n)^{\frac{m-1}{2}}}{n^{\frac{m-1}{2m}}}
= \sum_{1 \leq j_{1} \leq \cdots \leq j_{m} \leq \pi(x)} \vert c_{j_{1} , \ldots , j_{m}} \vert \frac{\big(\log (p_{j_{1}} \cdots p_{j_{m}}) \big)^{\frac{m-1}{2}}}{(p_{j_{1}} \cdots p_{j_{m}})^{\frac{m-1}{2m}}} \\
\leq \sum_{j_{m}=1}^{\pi(x)} \frac{(m \log p_{j_{m}})^{\frac{m-1}{2}}}{p_{j_{m}}^{\frac{m-1}{2m}}} \sum_{1 \leq j_{1} \leq \ldots \leq j_{m-1} \leq j_{m}} \frac{\vert c_{j_{1} , \ldots , j_{m}} \vert }{(p_{j_{1}} \cdots p_{j_{m-1}})^{\frac{m-1}{2m}}}\\
\leq \sum_{j_{m}=1}^{\pi(x)} \frac{(m \log p_{j_{m}})^{\frac{m-1}{2}}}{p_{j_{m}}^{\frac{m-1}{2m}}}
\bigg(  \sum_{1 \leq j_{1} \leq \ldots \leq j_{m-1} \leq j_{m}} \vert c_{j_{1} , \ldots , j_{m}} \vert^{2}\bigg)^{\frac{1}{2}} \times \\
\times \bigg(  \sum_{1 \leq j_{1} \leq \ldots \leq j_{m-1} \leq j_{m}} \frac{1}{(p_{j_{1}} \cdots p_{j_{m-1}})^{\frac{m-1}{m}}}\bigg)^{\frac{1}{2}} \, ,
\end{multline*}
where the last step follows from the Cauchy-Schwarz inequality. We use now the fact that for $0<\alpha < 1$ (see \cite[Satz~4.2, p. 22]{Pr57})
\[
    \sum_{p \leq x} p^{-\alpha}\, \ll \frac{1}{1-\alpha}\, \frac {x^{1-\alpha}}{\log x}
\]
to bound the last factor. By taking $\alpha=\frac{m-1}{m}$,
\[
 \bigg(  \sum_{1 \leq j_{1} \leq \ldots \leq j_{m-1} \leq j_{m}} \frac{1}{(p_{j_{1}} \cdots p_{j_{m-1}})^{\frac{m-1}{m}}}\bigg)^{\frac{1}{2}}
\leq \bigg(  \sum_{j \leq j_{m}} \Big( \frac{1}{p_{j}} \Big)^{\frac{m-1}{m}} \bigg)^{\frac{m-1}{2}}
\ll \Big( m \frac{p_{j_{m}}^{\frac{1}{m}}}{\log p_{j_{m}}}\Big)^{\frac{m-1}{2}} \, .
\]
With this we have
\begin{align*}
 \sum_{n \leq x} \vert a_{n} \vert \frac{(\log n)^{\frac{m-1}{2}}}{n^{\frac{m-1}{2m}}}
\ll & m^{m-1}  \sum_{j_{m}=1}^{\pi(x)} \frac{(\log p_{j_{m}})^{\frac{m-1}{2}}}{p_{j_{m}}^{\frac{m-1}{2m}}}
\Big( \frac{p_{j_{m}}^{\frac{1}{m}}}{\log p_{j_{m}}}\Big)^{\frac{m-1}{2}} \bigg(  \sum_{1 \leq j_{1} \leq \ldots \leq j_{m-1} \leq j_{m}} \vert c_{\mathbf{j}} \vert^{2}\bigg)^{\frac{1}{2}}  \,.
\end{align*}
Finally, by  Lemma~\ref{Fred2} and  \eqref{bohr},  there exists  $C(\kappa)>0$ such that
\[
 \sum_{n \leq x} \vert a_{n} \vert \frac{(\log n)^{\frac{m-1}{2}}}{n^{\frac{m-1}{2m}}} \leq C(\kappa) m^{m-1} (2\kappa)^{m} \|P\| = C(\kappa) m^{m-1} (2\kappa)^{m} \|D\|_\infty\, .
\]
\end{proof}

\subsection{Proofs}

 \begin{proof}[Proof of the lower estimate in Theorem~\ref{main}] We fix some $x \geq 2$. By Proposition~\ref{reduction}
 we only have to control each $m$-homogeneous part, $L_{m}(x)$. Note first that if $1 \leq n \leq x$ is such that
$\Omega (n)=m$ we have that $2^{m} \leq n \leq x$, which gives $m\le \frac{\log x}{\log 2}$. Then $\mathcal{H}^{x,m}_\infty = \{ 0\} $, and hence  $L_{m}(x)=1$, for every  $m>\frac{\log x}{\log 2}$. Thus
 \begin{equation}\label{BohrRadiuslambdax}
\frac{1}{3}\,\,\,\min_{1\leq m\leq \frac{\log x}{\log 2}} L_{m}(x)\,\, \le \,\,L_{x}.
\end{equation}
By \eqref{o'flynn} we have $L_{1}(x)=1$ for every $x$. We fix then $m \geq 2$ and observe that,  for every $D=\sum_{n \leq x} a_{n} n^{-s}  \in  \mathcal{H}^{(x,m)}_\infty$ we have $a_1=a_2=a_3=a_5=a_7=0$. By
Proposition~\ref{balasubramanian}, for each $\kappa>1$ there exists  $C(\kappa)>0$ such that
\[
\sum_{n \leq x}|a_n|\leq C(\kappa) m^{m-1}(2\kappa)^{m}x^{\frac{m-1}{2m}}\|D\|_\infty\,.
\]
This, using \eqref{pentagono}, gives
\[
 m^{-1} x^{-\frac{m-1}{2m^2}} \ll \Big( C(\kappa) m^{m-1}(2\kappa)^{m}x^{\frac{m-1}{2m}}\Big)^{-1/m} \leq  L_{m}(x)\,.
\]
But the sequence  $\big(x^{-\frac{m-1}{2m^2}}\big)_{m=2}^\infty$ is increasing to 1 (recall that $x \geq 2$). This  implies that  for all $m \geq 3$
\[
 m^{-1} x^{-\frac{1}{9}} \ll   L_{m}(x)\,\,,
 \]
 and hence for all $3\leq m\leq \frac{\log x}{\log 2}$
\begin{equation} \label{7}
\frac{\sqrt[4]{\log x}}{x^{\frac{1}{8}}} \ll  {\frac{\log 2}{\log x}} \frac{1}{x^{\frac{1}{9}}} \ll L_{m}(x)\,.
\end{equation}
We finish our argument by handling the case $m=2$. We observe first that $f(t)= \frac{\sqrt{\log t}}{t^\frac{1}{4}}=e^{g(t)}$ with $g(t)=\frac{1}{2}\log \log t -\frac{1}{4}\log t$, $t\geq  2$. Since  $g'(t)=\frac{1}{2t}\frac{2-\log t}{2 \log t}$, we have
 that $f$ is strictly decreasing for $t>e^2$. Then the sequence $\big(\frac{\sqrt{\log n}}{n^{\frac{1}{4}}}\big)$  is strictly decreasing for $n\geq 8$. Thus there exists $A>0$ such that for every $2\leq n\leq x$ we have
$ \frac{\sqrt{\log x}}{x^\frac{1}{4}}\leq A \frac{\sqrt{\log n}}{n^{\frac{1}{4}}}\,.$
Applying again Proposition~\ref{balasubramanian} we see that for every $D \in  \mathcal{H}^{(x,2)}_\infty$
\[
\frac{\sqrt{\log x}}{x^\frac{1}{4}} \sum_{n \leq x}|a_n|\leq A \, C(\kappa) 8\kappa^2 \|D\|_\infty\,,
\]
and hence
\[
\frac{\sqrt[4]{\log x}}{x^{\frac{1}{8}}} \ll    L_{2}(x)\,.
\]
This equation combined with \eqref{7} and \eqref{BohrRadiuslambdax} proves the lower estimate.
\end{proof}

\begin{proof}[Proof of the upper estimate in Theorem~\ref{main}] 

By Proposition~\ref{reduction}  it  suffices to show that there is a constant $C>0$  such that for all $x$
\begin{equation} \label{desired}
L_{2}(x)  \le C \frac{\sqrt[4]{\log x}}{x^{\frac{1}{8}}}\,.
\end{equation}
According to \eqref{pentagono}, fix some  $x$ and assume that $r>0$ satisfies
\begin{equation} \label{assumption}
\sum_{n \leq x} |a_n| \leq r^{-2} \sup_{t  \in \mathbb{R}} \Big\vert \sum_{n \leq x} a_{n} n^{it} \Big\vert \,.
\end{equation}
for every Dirichlet polynomial
$\sum_{n \leq x} a_n n^{-s} \in \mathcal{H}^{(x,2)}_\infty$
We choose $q$ to be the biggest natural number  $\leq \frac{\pi (\sqrt{x})}{2}$. We take a $q \times q$ matrix $(a_{nk})_{n,k}$  for which  $\vert a_{nk} \vert = 1$ and
$\sum_{l} a_{ln} \overline{a}_{lk} = q \delta_{nk}$. We define the Dirichlet series
\[
\sum_{n,k=1}^{q}  a_{nk} \frac{1}{(p_{n}p_{q+k})^s}  \in \mathcal{H}^{(x,2)}_\infty\,.
\]
Note that for every $1 \leq n,k\leq q$ we have $p_{n}p_{q+k} \leq p_{2q}^{2}   \le p_{\pi(\sqrt{x})}^2 \leq x$ and the Dirichlet series indeed belongs to $ \mathcal{H}^{(x,2)}_\infty$. Obviously, we have
\[
\sum_{n,k=1}^{q}  \big| a_{nk}  \big| = q^2\,.
\]
On the other hand,
\begin{multline*}
\sup_{t \in \mathbb{R}}\Big| \sum_{n,k=1}^{q}   a_{nk} p_{n}^{it}  p_{q+k}^{it} \Big|
\leq  q^{1/2} \Big(\sum_{k}  \Big|\sum_{n}  a_{nk} p_{n}^{it} \Big|^2\Big)^{1/2} \\
=  q^{1/2} \Big(  \sum_{k}  \sum_{n_1,n_2}   a_{k n_{1}}  \overline{a_{k n_{2}}} p_{n_{1}}^{it}  p_{n_{2}}^{-it}    \Big)^{1/2}
=  q^{1/2} \Big( \sum_{n_1,n_2} p_{n_{1}}^{it}  p_{n_{2}}^{-it}   \sum_{k}     a_{k n_{1}}  \overline{a_{k n_{2}}} \Big)^{1/2}    \\
= q^{1/2} \Big( \sum_{n_1,n_2} p_{n_{1}}^{it}  p_{n_{2}}^{-it}  q \delta_{n_1,n_2} \Big)^{1/2}
= q \Big( \sum_{n} |p_{n}^{it}|^2  \Big)^{1/2}  \leq  q^{3/2}\,.
\end{multline*}
Then by \eqref{assumption} we conclude $q^{2} \leq r^{-2} q^{\frac{3}{2}}$.
But from the  prime number theorem we deduce that there is a (universal) constant $C>0$ such  that $\frac{\sqrt{x}}{\log x} \le C q \,,$
and therefore
\[
r \leq C   \frac{\sqrt[4]{\log x}}{x^{\frac{1}{8}}}\,.
\]
Clearly, this gives the desired estimate \eqref{desired}.
\end{proof}

\section{Dirichlet-Bohr radii} \label{generales}
The main goal of  the previous section was to find the correct asymptotic order of  the Dirichlet-Bohr radius $L\big( \big\{n \in \mathbb{N}\,|\, 1 \leq n \leq x  \big\} \big)$.

Analysing  the  ideas of our  proof, we in the coming subsection show how to reduce the study of  Dirichlet-Bohr radii $L(J)$ for  index sets
to the study of Bohr radii for holomorphic functions in infinitely many variables with lacunary monomial coefficients. Finally, we  treat a series of old and new  examples.

\subsection{Reduction II}
Let  $\Lambda$ be a subset of  $\mathbb{N}_0^{(\mathbb{N})}$. Consider the Banach space
\begin{align*}
H^{\Lambda}_\infty(B_{c_0}) :=  \Big\{ f \in H_\infty(B_{c_0}) \,\,\Big| \,\, c_\alpha(f)\neq 0 \,\,\, \text{ only if }
\,\,\,   \alpha \in \Lambda
    \Big\}\,,
\end{align*}
where as usual $H_\infty(B_{c_0}) $ denotes the Banach space of all bounded holomorphic (= Fr\'echet differentiable) functions
on the open unit ball $B_{c_0}$ of the Banach space of all null sequences $c_0$.\\

Now, the \textit{Bohr radius} $K(\Lambda)$ is defined to be the  best $ r=r(\Lambda)\ge 0 $ such that for every $ f \in H^\Lambda_\infty(B_{c_0}) $ we have
\[
\sum_{\alpha\in \Lambda}|c_\alpha(f)| r^{|\alpha|}\le \|f\|_\infty\,.
\]

Note that, with this notation, the classical Bohr radius $K_{n}$ is just $K(\mathbb{N}_{0}^{n})$.

The following  result extends \eqref{count1} to arbitrary index sets. Let us note that the proof of  \eqref{count1} was based on Bohr's fundamental lemma  \eqref{bohr}. We need, then, an extension of this. Inspired by an idea of Bohr and based
on the fundamental theorem of arithmetics we here consider the following bijection:
\[
\mathfrak{b} : \,\,\mathbb{N}_0^{(\mathbb{N})} \rightarrow  \mathbb{N}\,, \,\,\, \mathfrak{b}(\alpha)=p^\alpha \,.
\]
We denote now by $\mathcal{H}_\infty$ all Dirichlet series $\sum_n a_n n^{-s}$ defining a bounded holomorphic function on $[\re  s > 0]$; this vector space together with the sup norm  on $[\re  s > 0]$ forms a Banach space.
By \cite[Lemma~2.3 and Theorem~3.1]{HeLiSe97} (a fact also essentially due to Bohr \cite{Bo13_Goett}) there is  a unique
isometric and linear bijection $\Phi$ from $H_\infty(B_{c_0})$ onto $\mathcal{H}_\infty$ such $\Phi(z^\alpha) =  n^{-s}$ with $\mathfrak{b}(\alpha)=n$:
\[
H_\infty(B_{c_0}) \,\,=\,\,\mathcal{H}_\infty\,.
\]
Using this general principle a simple translation argument from Dirichlet series into holomorphic functions, and vice versa gives the following result.

\begin{proposition} \label{BDB}
For each set
$J \subset \mathbb{N}$ and  $\Lambda \subset \mathbb{N}_0^{(\mathbb{N})}$  with $J=\mathfrak{b}(\Lambda)$
\[
K(\Lambda) = L(J)\,.
\]
\end{proposition}

\noindent Our next device reduces the estimation of Dirichlet-Bohr radii of a given  index set $J$ to the estimation of  Dirichlet-Bohr radii of certain parts of $J$.
Given  $J \subseteq \mathbb{N}$ and $n,m \in \mathbb{N}$,  the  \textit{$n$-dimensional kernel of $J$} is defined to be
\[
J(n) = \big\{ k \in J \,\, \big |  \,\, \forall j > n : p_j \nmid k \big\}\,,
\]
and its  \textit{$m$-homogeneous kernel}
\[
J[m] = \big\{ k \in J \,\, \big |  \,\, \Omega(k)=m  \big\}\,.
\]
Note that when $J=\mathbb{N}$, then the $n$-dimensional kernel consists of all the natural numbers that factor through the first $n$ primes   and the $m$-homogeneous kernel consists of those which have precisely $m$ prime divisors (counted with multiplicities). In other words
\begin{align*}
\mathbb{N}(n) = \{p_{1}^{\alpha_{1}}\cdots p_{n}^{\alpha_{n}} \big| \alpha \in \mathbb{N}_{0}^{n}  \}  &&
\text{and} &&
\mathbb{N}[m] = \{p_{1}^{\alpha_{1}}\cdots p_{k}^{\alpha_{k}} \cdots  \big| \alpha_{1} + \cdots + \alpha_{k} + \cdots =m \}  \, .
\end{align*}
Then, clearly $J(n)=J\cap \mathbb{N}(n)$ and $J[m] = J \cap \mathbb{N}[m]$. We also have
\begin{align*}
\mathfrak{b}^{-1}(J(n))  = \big\{  \alpha \in \mathbb{N}_0^n \,\big|\, p^\alpha \in J\big\}
&& \text{and} &&
\mathfrak{b}^{-1}(J[m])  = \big\{  \alpha \in \mathbb{N}_0^{(\mathbb{N})}\,\,\big|\, p^\alpha \in J  \text{ with } \,\, |\alpha|=m \big\}\,.
\end{align*}
In particular, $\mathfrak{b}^{-1}(\mathbb{N}(n))  = \mathbb{N}_0^{n}$ and $\mathfrak{b}^{-1}(\mathbb{N}[m])  = \big\{  \alpha \in \mathbb{N}_0^{(\mathbb{N})}\,\,\big|\, \, |\alpha|=m \big\}$\,.
Let us finally observe that
\[
\mathbb{N}(n)[m] = \{p_{1}^{\alpha_{1}}\cdots p_{n}^{\alpha_{n}} \big| \alpha \in \mathbb{N}_{0}^{n} \text{ and } \alpha_{1} + \cdots + \alpha_{n} =m   \} = \mathbb{N}[m] (n)
\]
and from this $J(n)[m] = J \cap \mathbb{N}(n)[m]  = J \cap \mathbb{N}[m] (n) = J[m](n)$ for every $J \subseteq \mathbb{N}$ and every $n,m$. We can now give our announced reduction device.
\begin{proposition} \label{reduction2}
Let $J$ be a subset of \,$\mathbb{N}$. Then
\vspace{2mm}
\begin{enumerate}
\item $L(J)\,\, = \,\,\inf_{n}  L(J(n))$  \label{duke}\\
\item $\frac{1}{3} \inf_{m}  L(J[m]) \,\leq \,\, L(J) \,\, \leq \,\,\inf_{m}  L(J[m]) $ \label{count}
\end{enumerate}
\end{proposition}

\begin{proof}
The proof of the second statement follows from a word by word copy of the proof of Proposition
\ref{reduction}. The argument of the first statement is easy after a translation to holomorphic functions by
Proposition \ref{BDB}.
\end{proof}

Of course, (\ref{duke}) and (\ref{count}) can be combined showing the  infimum over $\big(  L(J[m](n) ) \big)_{m,n}$ and
$\big(  L(J(n)[m])\big)_{m,n}$, respectively,  up to the constant $1/3$ equals $L(J)$.

\vspace{3mm}

\subsection{Examples}
\noindent We first recover with this systematic language the fundamental examples \eqref{o'flynn} that were already mentioned in the introduction.

\begin{example} \label{mother} \text{ }
\begin{enumerate}
\item $L\big(\mathbb{N}[1]\big)=L\big(\big\{ p \,|\, p \, \text{ prime }  \big\}\big) = 1$ \label{mother i}\\
\item $L\big(\mathbb{N}(1)\big)= L\big(\big\{ 2^k \,|\, k \in \mathbb{N}      \big\}\big) = \frac{1}{3} $ \label{mother ii}
\end{enumerate}
\end{example}
\noindent We remark that (\ref{mother i}) here is nothing else than Bohr's inequality \eqref{Bohr's inequality}, whereas (\ref{mother ii})  is just a reformulation via Proposition~\ref{BDB} of Bohr's power series theorem  \eqref{Bohr's power series theorem}
 (see also \eqref{motherAAA}). Basically, these and the one in the following example are the only precise values of Dirichlet-Bohr radii we know.

\begin{example} \label{another interesting example}
$ \displaystyle L\Big(\big\{ p_\ell^k \,\big|\,k, \ell \in \mathbb{N}  \big\}\Big) = \frac{1}{3}$.
\end{example}

\noindent This turns out to be  an immediate consequence of the following more general result.
Given a subset $A$   of $\mathbb{N}$, we will denote its cardinal number by $|A|$.

\begin{proposition} \label{manolin}
Let $P_k\,,\, k \in \mathbb{N}$ be disjoint  sets of primes such that
\[
n=\max_k | P_k| < \infty\,.
\]
Define $J_{P_k}$ to be the set of all natural numbers which are finite products of primes in $P_k$, that is
\[
J_{P_k} =\big\{ p^\alpha \,| \, \alpha_{j} = 0,\, {\text if} \ p_j \notin  P_k \big\}\,.
\]
Then
\[
L\Big(\bigcup_k J_{P_k} \Big)  = L\big(\mathbb{N}(n)\big)\,.
\]
\end{proposition}

\noindent Clearly, Example \ref{another interesting example} is an immediate consequence of this result: put $P_k = \{p_k\}$ (the $k$-th prime) and apply Example~\ref{mother} together with Proposition~\ref{manolin}.

\begin{proof}
Define the sets  $\Lambda_k = \mathfrak{b}^{-1}(J_{P_k}) \subset \mathbb{N}_0^{(\mathbb{N})}$, and
recall that  $\mathbb{N}_0^n = \mathfrak{b}^{-1}(\mathbb{N}(n))$. Looking at Proposition \ref{BDB} it suffices to prove that
\[
K\big( \bigcup_k \Lambda_k \big) = K\big(\mathbb{N}_0^n \big)\,.
\]
Let
$I_k =\bigcup_{\alpha \in \Lambda_k }  \text{supp} \alpha \subset \mathbb{N}$
be the support of $\Lambda_k$. Clearly, we have $n_k:=|I_k|=|P_k|$ for all $k$.
We identify
$\spanned  \{e_i: i\in I_k\}$ with $\mathbb{C}^{n_k}\,.$\\
By considering bounded holomorphic functions with support in any  $I_k$ of length $n$, we get that $K\big( \bigcup_k \Lambda_k \big)\leq K\big(\mathbb{N}_0^n \big)$. We have to prove now the reverse inequality
\begin{equation} \label{jonny}
K\big(\mathbb{N}_0^n \big) \,\,\leq \,\,K\big( \bigcup_k \Lambda_k \big)\,.
\end{equation}
Now, we want to show that
\[
\sum_{\alpha\in \bigcup_k \Lambda_k }|a_\alpha|\, K\big(\mathbb{N}_0^n \big)^{|\alpha|}\le \sup_{z\in B_{c_0}}\Big | \sum_{\alpha\in \bigcup_k \Lambda_k } a_\alpha z^\alpha \Big |\,
\]
for every function $\sum_{\alpha\in \bigcup_k \Lambda_k } a_\alpha z^\alpha \in H_\infty(B_{c_0})$. Since the $\Lambda_k$'s are disjoint, we have
$$\sup_{z\in B_{c_0}}\Big |  \sum_{\alpha\in \bigcup_{k=1}^N \Lambda_k } a_\alpha z^\alpha \Big |
\le \sup_{z\in B_{c_0}}\Big |  \sum_{\alpha\in \bigcup_k \Lambda_k } a_\alpha z^\alpha \Big |$$
for all $N$, and then it will be enough to show that
\begin{equation}\label{finiteunion}
\sum_{\alpha\in \bigcup_{k=1}^N }|a_\alpha|\, K\big(\mathbb{N}_0^n \big)^{|\alpha|}\le\sup_{z\in B_{c_0}}\Big |  \sum_{\alpha\in \bigcup_{k=1}^N \Lambda_k } a_{\alpha} z^{\alpha} \Big |\,.
\end{equation}
We proceed now by induction on $N$.
For $N=1$, \eqref{finiteunion} is just a consequence of the following: $K(\mathbb N_0^n)\le K(\mathbb N_0^{n_1})=K(\Lambda_1)$.
For the inductive step, we write
$$ \sum_{\alpha\in \bigcup_{k=1}^N \Lambda_k } a_{\alpha} z^{\alpha} = a_0 + f_1(u_1) + \cdots + f_N(u_N)\,,$$ where
$u_k$ is the projection of $z$ in the $\Lambda_k$-coordinates
and
$$f_k(w)= \sum_{\substack{\alpha\in \mathbb{N}_0^{n_k}\\ |\alpha|\ge 1}} a^k_{\alpha} w^{\alpha}\,$$
for $w\in \mathbb C^{n_k}$. Note that $f_k(0)=0$ for every $k$.
By inductive hypothesis we know that
\begin{equation}\label{HI}
|a_0|+\sum_{k=1}^{N-1}\sum_{\substack{\alpha\in \mathbb{N}_0^{n_k}\\ |\alpha|\ge 1}} |a_{\alpha}|\, K(\mathbb N_0^n )^{|\alpha|}
\leq \sup_{u_1 \in \mathbb{D}^{n_1},\ldots, u_{N-1} \in \mathbb{D}^{n_{N-1}}}\Big|a_0 + \sum_{k=1}^{N-1}f_k(u_k)\Big|.
\end{equation}
Fix now $u_k \in \mathbb{D}^{n_k}$ for $k=1,\ldots, N-1$  and set $\widetilde a_0 = a_0+\sum_{k=1}^{N-1}f_k(u_k)$.
Since $K(\mathbb N_0^n)\le K(\mathbb N_0^{n_N})=K(\Lambda_N)$,
we have
\[
\big| \widetilde a_0 \big|+ \sum_{\substack{\alpha\in \mathbb{N}_0^{n_N}\\ |\alpha|\ge 1}} |a^N_{\alpha}|\, K(\mathbb N_0^n)^{|\alpha|}  \le   \sup_{u_{N} \in \mathbb{D}^{n_{N}}}\Big| \widetilde a_0 + f_N(u_N)\Big|\,,
\]
which just means that
\begin{equation}\label{HI2}
\big| a_0+\sum_{k=1}^{N-1}f_k(u_k)\big|+ \sum_{\substack{\alpha\in \mathbb{N}_0^{n_N}\\ |\alpha|\ge 1}} |a^N_{\alpha}|\, K(\mathbb N_0^n)^{|\alpha|}  \le   \sup_{u_{N} \in \mathbb{D}^{n_{N}}}\Big| \big(a_0+\sum_{k=1}^{N-1}f_k(u_k) \big) + f_N(u_N)\Big|\,.
\end{equation}
Combining \eqref{HI} and \eqref{HI2} we obtain \eqref{finiteunion}.
\end{proof}

\noindent In the following results we present asymptotically correct estimates on Dirichlet-Bohr radii.

\begin{example} \label{motherAA} \text{ }
\begin{itemize}
\item[(1)]
$\lim_n \frac{L\big(\mathbb{N}(n)\big)}{\sqrt{\frac{\log n}{n}}} =1$\,;
\item[(2)] There is a constant $C >1$ such that
\begin{align*}
C^{-m}   \Big( \frac{m}{n}\Big)^{\frac{m-1}{2m}}
\leq
&
L\Big( \big(\mathbb{N}(n)\big)[m]\Big)
\leq
C^{m}  \Big( \frac{m}{n}\Big)^{\frac{m-1}{2m}} \quad \,\,\text{ for  }  \quad  n > m
\\
C^{-m}
\leq
&
L\Big( \big(\mathbb{N}(n)\big)[m]\Big)
\leq
C^{m}
\quad
\quad \quad \quad \quad
\text{ for  }  \quad  n \leq m\,.
\end{align*}
\end{itemize}
\end{example}
\noindent
Both results follow from Proposition \ref{BDB} and their counterparts for Bohr radii:
\begin{equation*}
\lim_n \frac{K\big( \mathbb{N}_0^n  \big)}{\sqrt{\frac{\log n}{n}}} =1
\end{equation*}
and
\begin{align*}
C^{-m}   \Big( \frac{m}{n}\Big)^{\frac{m-1}{2m}}
\leq
&
K
\Big(
\big\{ \alpha \in \mathbb{N}_0^n \, \big| \, |\alpha| =m    \big\}
\Big)
\leq
C^{m}  \Big( \frac{m}{n}\Big)^{\frac{m-1}{2m}} \quad \,\,\text{ for  }  \quad  n > m
\\
C^{-m}
\leq
&
K
\Big(
\big\{ \alpha \in \mathbb{N}_0^n \, \big| \, |\alpha| =m    \big\}
\Big)
\leq
C^{m}
\quad
\quad \quad \quad \quad
\text{ for  }  \quad  n \leq m\,.
\end{align*}
The first formula is due to Bayart, Pellegrino, and Seoane-Sep\'ulveda \cite{BaPeSe13}, who improve an earlier result from \cite{DeFrOrOuSe11}). The upper estimate in the second  result follows  from \cite{DeFrOrOuSe11}, and the lower one is a  consequence of the Kahane-Salem-Zygmund inequality (or  \cite[Lemma 2.1 and (4.4)]{DeGaMa03}). It would be of particular interest
to know the precise values of $L\big(\mathbb{N}(n)\big)\,, $  $L\big(\mathbb{N}[m])\big) $ and $L\big( \big(\mathbb{N}(n)\big)[m]\big)$
for all/some $n,m >1$.

If Example \ref{motherAA} is combined with Proposition \ref{reduction2}, then we see the following examples.

\begin{example} \label{final example} \text{ }
\begin{itemize}
\item[(1)]
$L\big( \mathbb{N}  \big)=0$  \text{ }\,;\\
\item[(2)]
$L\big( \mathbb{N}[m] \big) =0\,\,\, \text{ for all  } \,\,\,m >1$\,.
\end{itemize}
\end{example}

\end{document}